\documentclass{amsart}
\usepackage{graphicx}
\usepackage{newlattice}

\DeclareMathOperator{\length}{length}

\newcommand{\fp}{\F p}
\newcommand{\fq}{\F q}

\newcommand{\oa}{\ol \bga}

\theoremstyle{plain}
\newtheorem{theorem}{Theorem}

\newtheorem{lemma}[theorem]{Lemma}
\newtheorem{corollary}[theorem]{Corollary}

\theoremstyle{definition}

\begin{document}
\title[Congruences of the fork extensions. II. The congruence gamma]
{Congruences of the fork extensions. II.\\
The congruence gamma}  
\author{G. Gr\"{a}tzer} 
\address{Department of Mathematics\\
  University of Manitoba\\
  Winnipeg, MB R3T 2N2\\
  Canada}
\email[G. Gr\"atzer]{gratzer@me.com}
\urladdr[G. Gr\"atzer]{http://server.maths.umanitoba.ca/homepages/gratzer/}

\date{\today}

\begin{abstract}
G.~Cz\'edli and E.\,T.~Schmidt introduced in 2012 
the fork extension.
Continuing from Part I,  we investigate 
the congruences of a fork extension.
\end{abstract}

\maketitle

\section{Introduction}\label{S:Introduction}

\begin{figure}[b!]
\centerline{\includegraphics{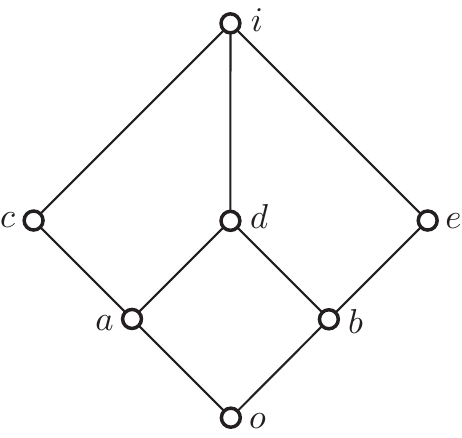}}
\caption{The lattice $\SfS 7$}\label{F:s7}
\end{figure}

Let $L$ be a slim, planar, semimodular lattice, an \emph{SPS lattice}.
As in G.~Cz\'edli and E.\,T.~Schmidt~\cite{CSb}, 
\emph{inserting a fork} to $L$ at the covering square~$S$, 
firstly, replaces~$S$ by a~copy of $\SfS 7$ 
(see the lattice $\SfS 7$ in Figure~\ref{F:s7}). 

Secondly, if there is a chain 
$u\prec v\prec w$ such that~the element $v$ has just been added 
and 
$
   T = \set{x=u \mm z, z, u, w=z \jj u}
$
is a covering square in the lattice~$L$ 
(and so $u \prec v \prec w$ is not on the boundary of $L$) 
but $x \prec z$ at the present stage of the construction,
then we insert a new element~$y$ 
such that $x \prec y \prec z$ and $y \prec v$.

Let $L[S]$ denote the lattice we obtain when the procedure terminates. 
We say that $L[S]$ is obtained from
$L$ by \emph{inserting a fork to $L$} at the covering square~$S$.
See Figure~\ref{F:forks} for an illustration.

\begin{figure}[hbt]
\centerline{\includegraphics{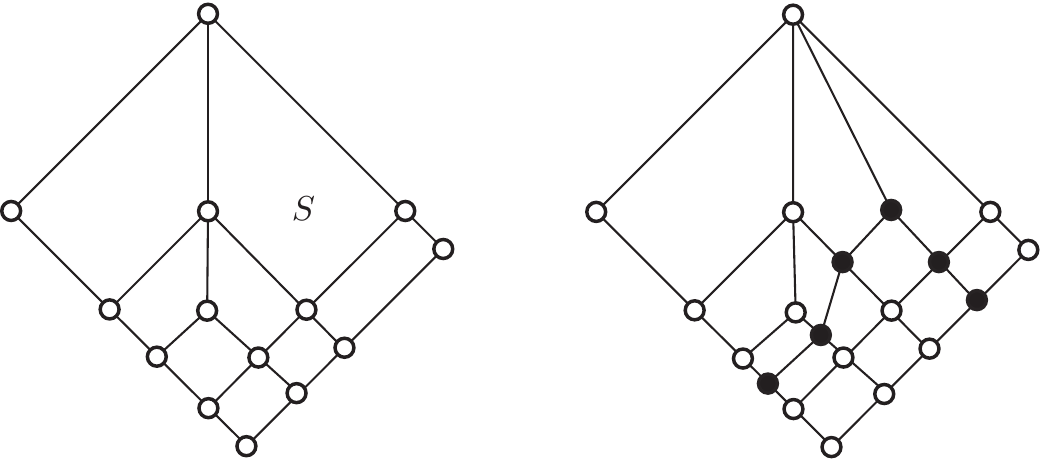}}
\caption{Inserting a fork at $S$}\label{F:forks}
\end{figure}

Let $L$ be an SPS lattice
and $S$ a covering square of $L$.
In~this series of research notes we examine the connections between the congruence lattice of $L$ and the congruence lattice of $L[S]$.
In Part I, G. Gr\"atzer \cite{gGa}, we proved the following result:

\begin{theorem}\label{T:extension}
Let $L$ be a slim, semimodular, planar lattice. 
Let $S$ be a covering square of $L$.
Then every congruence of $L$ extends to $L[S]$.
\end{theorem}

Let $L$ be an SPS lattice 
and let $S = \set{o, a_l, a_r, i}$ be a covering square of~$L$.
We call $S$ a \emph{tight square} 
if $i$ covers exactly two elements, namely, $a_l$ and $a_r$, in $L$;
otherwise, $S$ is a \emph{wide square}.

\begin{theorem}\label{T:wide}
Let $L$ be an SPS lattice. 
Let $S$ be a wide square. 
Then $L[S]$ is a congruence-preserving extension of $L$.
\end{theorem}

Let $\bga$ be a congruence of $L$. 
Then there is a smallest congruence $\oa$ of $L[S]$ generated by $\bga$.
By Theorem~\ref{T:extension}, $\bga$ is the restriction of $\oa$ to $L$.
Moreover, if $\bga$ is a join-irreducible congruence of $L$, 
then $\oa$ is a join-irreducible congruence of $L[S]$.

\begin{theorem}\label{T:tight}
Let $L$ be an SPS lattice. 
Let $S$ be a \emph{tight square}. 
Then $L[S]$ has exactly one join-irreducible congruence, $\bgg(S)$,
that is not the minimal extension of a join-irreducible congruence of $L$.
\end{theorem}

In this note we will examine the congruence $\bgg(S)$.

We will use the notations and concepts of lattice theory, 
as in \cite{LTF}.

\section{Congruences of finite lattices}\label{S:Congruences}

As~illustrated in Figure~\ref{F:cong2}, 
we say that $[a,b]$ is \emph{up congruence-perspective} 
to $[c,d]$ and write $[a,b] \cperspup [c,d]$\label{GoN:cperspup} 
if $a \leq c$ and $d = b \jj c$; similarly,
$[a,b]$ is \emph{down congruence-perspective} to $[c,d]$ 
and write $[a,b] \cperspdn [c,d]$ if $d \leq b$ and $c = a \mm d$.
If $[a,b] \cperspup [c,d]$ \emph{or} $[a,b] \cperspdn [c,d]$, 
then $[a,b]$ is \emph{congruence-perspective} to $[c,d]$ 
and we write $[a,b] \cpersp [c,d]$.

\begin{figure}[tbh]
  \centerline{\includegraphics{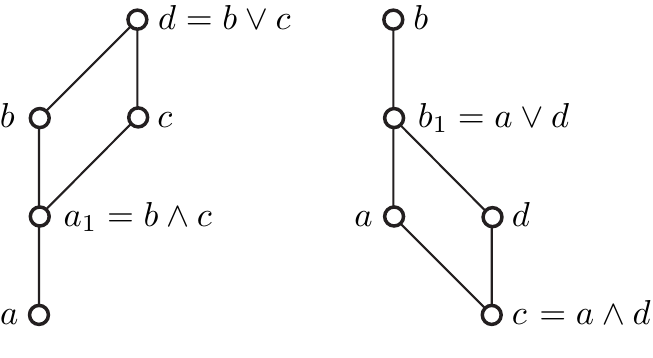}}
  \caption{$[a,b] \cperspup [c, d]$ and $[a,b] \cperspdn [c, d]$}\label{F:cong2}
\end{figure}

If for some natural number~$n$ 
and intervals $[e_i, f_i]$, for $0 \leq i \leq n$,
\[
   \inv{a}{b} = \inv{e_0}{f_0} \cpersp \inv{e_1}{f_1} \cpersp 
      \cdots \cpersp \inv{e_n}{f_n} = \inv{c}{d},
\]
then we call $[a,b]$ \emph{congruence-projective}
to $[c,d]$, and we write $[a, b] \cproj [c, d]$. 

We now state a classic result in a special case.

\begin{lemma}\label{L:cproj}
Let L be a lattice, $a \leq b$ in $L$, and let $\fq$ be a prime interval.
Then $\fq$ is collapsed by $\con{a, b}$
if{f} $[a, b] \cproj \fq$.
In fact, there is a prime interval $\fp$ in $[a,b]$ 
such that $\fp \cproj \fq$.
\end{lemma}

There are two 
technical lemmas we need.

\begin{lemma}\label{L:technical}
Let $L$ be a finite lattice. 
Let $\bgd$ be an equivalence relation on $L$
with intervals as equivalence classes.
Then $\bgd$ is a congruence relation if{}f the following condition 
and its dual hold:
\begin{equation}\label{E:cover}
\text{If $x$ is covered by $y \neq z$ in $L$ 
and $x \equiv y\,(\tup{mod}\, \bgd)$,
then $z \equiv y \jj z\,(\tup{mod}\,\bgd)$.}\tag{C${}_{\jj}$}
\end{equation}
\end{lemma}

\begin{proof}
We want to prove that if $x \leq y$ and $\cng x=y (\bgd)$,
then $\cng x \jj z = y \jj z(\bgd)$.
The proof is a trivial induction first on $\length[x, y]$
and then on $\length[x, x \jj z]$.
\end{proof}

Let (C${}_{\mm}$) denote the dual of (C${}_{\jj}$).

The following statement is a variant of Lemmas 245 and 246 of \cite{LTF}.

\begin{lemma}\label{L:technical2}
Let $L$ be a finite lattice. Let $K$ be an extension of $L$.
Let us assume that every congruence of $L$ extends to $K$.
Then $K$ is a congruence-preserving extension if{}f the following condition is satisfied:
\begin{enumerate}\label{E:prime}
\item[\tup{(P)}]for every prime interval $\fp$ of $K$,
$\fp$ not in $L$,
there exists a prime interval~$\fq$ of $L$ 
such that
$\con{\fp} = \con{\fq}$ in $K$.
\end{enumerate}
\end{lemma}

\section{The fork construction}\label{S:forks}

Let $L$ be an SPS lattice. 
Let $S =\set{o, a_l, a_r, i}$ be a covering square of $L$,
let $a_l$ be to the left of $a_r$
We need some notation for the $L[S]$ construction, 
see Figure~\ref{F:forksdetails}. 

\begin{figure}[h]
\centerline{\includegraphics{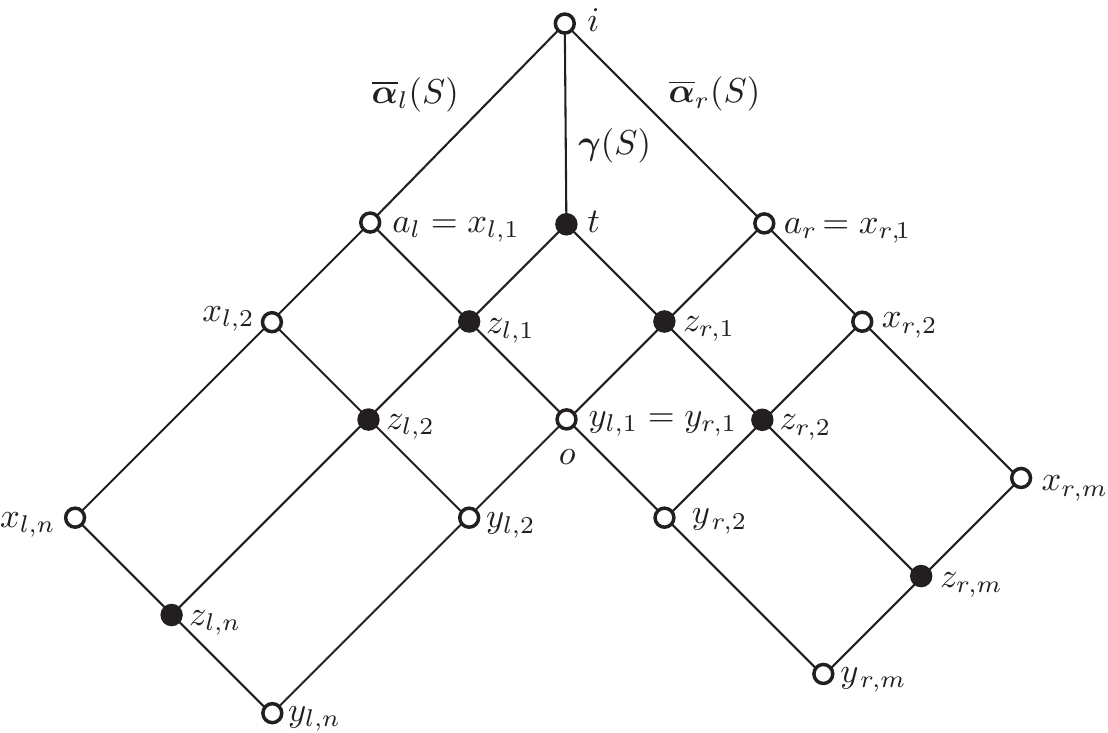}}
\caption{Notation for the fork construction}\label{F:forksdetails}
\end{figure}

We start the construction 
by adding the elements $t$, $z_{l,1}$, and $z_{r,1}$
so that the set $\set{o, z_{l,1}, z_{r,1}, a_l, a_r, t,i}$ 
forms a sublattice $\SfS 7$. We say that this $\SfS 7$ sublattice and the covering square $S$ are \emph{associated} with each other.  

Let $a_l = x_{l,1}$, $o = y_{l,1}$. 
If $k$ is the largest number so that 
$x_{l,k}$, $y_{l,k}$, and $z_{l,k}$ have already been defined, and
\[
   T = \set{y_{l,k+1} = x_{l,k+1} \mm y_{l,k}, x_{l,k+1}, y_{l,k},
    x_{l,k} = x_{l,k+1} \jj y_{l,k}}
\]
is a covering square in $L$, then we add the element $z_{l,k+1}$,
so we get two new covering squares 
$\set{y_{l,k+1}, x_{l,k+1}, y_{l,k}, x_{l,k}}$
and $\set{z_{l,k+1}, y_{l,k+1}, z_{l,k}, y_{l,k}}$.
We proceed similarly on the right.

So $L[S]$ is constructed by inserting the elements in the set
\[
   F[S] = \set{t,z_{l,1} \succ \dots \succ z_{l,n}, 
          z_{r,1} \succ \dots \succ z_{r,m}}
\]
so that $\set{o, z_{l,1}, z_{r,1}, a_l, a_r, t,i}$ 
is a sublattice $\SfS 7$,
moreover, $x_{l,i} \succ z_{l,i} \succ y_{l,i}$ for $i = 1, \dots, n$, and $x_{r,i} \succ z_{r,i} \succ y_{r,i}$ for $i = 1, \dots, m$.
The new elements are black filled in Figure~\ref{F:forksdetails}.

We name a few join-irreducible congruences of $L$ and $L[S]$ that will plays an important role: $\bga_l(S) = \consub{L}{a_l,i}$ and  
$\bga_r(S) = \consub{L}{a_r,i}$ in $L$ and 
 $\ol\bga_l(S) = \consub{L[S]}{a_l,i}$,   
$\ol\bga_r(S) = \consub{L[S]}{a_r,i}$ and $\bgg(S) = \consub{[L[S]}{t,i}$. 
The following results are well known.

\begin{lemma}\label{L:known}
Let $L$ be an SPS lattice. 
\begin{enumeratei}
\item An element of $L$ has at most two covers.

\item Let $a \in L$. 
Let $a$ cover the three elements $x_1$, $x_2$, and $x_3$.
Then the set $\set{x_1,x_2,x_3}$ generates an $\SfS 7$ sublattice.

\item If the elements $x_1$, $x_2$, and $x_3$ are adjacent, 
then the $\SfS 7$ sublattice of \tup{(ii)} is a cover-preserving sublattice.
\end{enumeratei}
\end{lemma}
Finally, we state the Structure Theorem for SPS Lattices of
G. Cz\'edli and E.\,T. Schmidt~\cite{CSa}:

\begin{theorem}\label{T:Structure Theorem}
Let $L$ be an SPS lattice. 
There exists a planar distributive lattice~$D$
such that $L$ can be obtained from $D$ by a series 
of fork insertions.
\end{theorem}

\section{Wide squares}\label{S:wide}

\begin{lemma}\label{L:moretwo}
Let $L$ be an SPS lattice. 
Let $S$ be a \emph{wide square}. 
Then $\bgg(S)  = \ol\bga_l(S)$ 
or $\bgg(S)  = \ol\bga_r(S)$ in $L[S]$.
\end{lemma}

\begin{proof}
Since $S$ is wide, 
the element $i$ covers an element $a$ in $L$, with $a \neq a_l, a_r$.
Either $a$ is to the left of $a_l$ or to the right of $a_r$,
let us assume the latter.
By~Lemma~\ref{L:known}, the set $\set{a_l,a_r,a}$ 
generates an $\SfS 7$ sublattice in $L$.

Then $\con{i,t} \leq \con{a_r,i}$, 
computed in the $\SfS 7$ sublattice generated by $\set{a_l,t,a_r}$
and $\con{i,t} \geq \con{a_r,i}$, 
computed in the $\SfS 7$ sublattice generated by $\set{a_r,t,a}$,
yielding $\bgg(S)  = \ol\bga_r(S)$.

The symmetric case yields $\bgg(S)  = \ol\bga_l(S)$.
\end{proof}

\begin{figure}[tbh]
  \centerline{\includegraphics{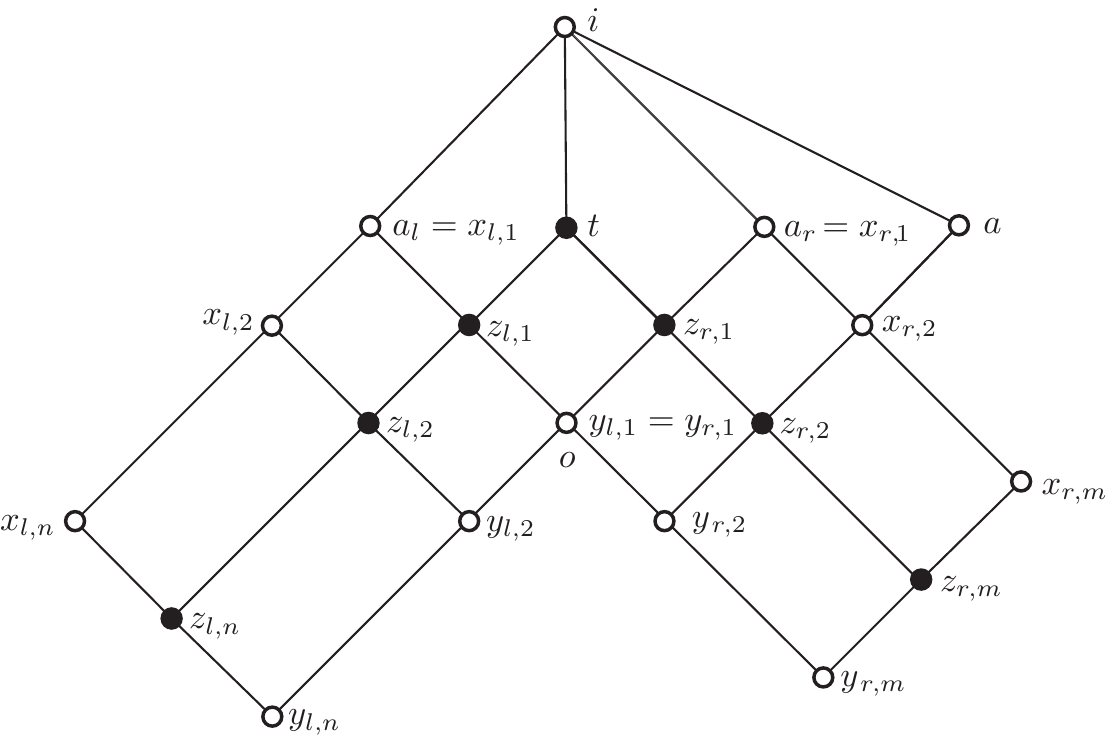}}
  \caption{Wide square}\label{F:widesquare}
\end{figure}

Now we prove Theorem~\ref{T:wide}.

By Theorem~\ref{T:extension}, 
the assumptions of Lemma~\ref{L:technical2}
hold for $L$ and $K = L[S]$. So~by~(P) of Lemma~\ref{L:technical2},
we have to show that for every prime interval $\fp$ of $L[S]$,
if~$\fp$ is not in~$L$, then
there exists a prime interval~$\fq$ of $L$ such that
$\con{\fp} = \con{\fq}$ in $L[S]$.

The prime intervals of $L[S]$ that are not in $L$ 
can be listed as follows:
\begin{align}
   [z_{l,1}, x_{l,1}],& \dots, [z_{l,n}, x_{l,n}];\label{E:set1}\\ 
   &[t, i]; \label{E:set2}\\
   [z_{r,1}, x_{r,1}],& \dots, [z_{r,m}, x_{r,m}];\label{E:set3}\\ 
   [y_{l,1}, z_{l,1}],& \dots, [y_{l,n}, z_{l,n}]; \label{E:set4}\\
   [y_{r,1}, z_{r,1}],& \dots, [y_{r,m}, z_{r,m}].\label{E:set5}
\end{align}

We choose the following prime intervals $\fq$ of $L$ to satisfy (P)
of Lemma~\ref{L:technical2}:

\begin{enumeratei}
\item for the prime intervals 
in the lists \eqref{E:set1}--\eqref{E:set3} and \eqref{E:set5}, 
choose $\fq = [a_l, i]$;

\item for the prime intervals in the list \eqref{E:set4}, 
choose $\fq = [a_r, i]$.
\end{enumeratei}

\section{Distributive covering squares}\label{S:Distributive}

Let $L$ be an SPS lattice. 
Let $S =\set{o, a_l, a_r, i}$ be a covering square of $L$. 
We call~$S$ \emph{distributive} if the ideal generated by $S$,
that is, $\id{i}$, is distributive.

Let $L$ be an SPS lattice. Let $A_1$ and $A_2$ be covering $\SfS 7$ sublattices of~$L$ 
in the covering squares $S_1$ and $S_2$, with unit elements $i_1$ and $i_2$, respectively. 
Let $A_2 < A_1$ mean that $i_2 < i_1$.
Let us call $A_1$ \emph{minimal} if $A_2 < A_1$
fails for any covering $\SfS 7$ sublattice $S_2$ of~$L$.

The following statement  
see G. Cz\'edli and E.\,T. Schmidt~\cite{CSa}.

\begin{lemma}\label{L:minimal}
Let $L$ be an SPS lattice 
and let $S$ be  be a covering square of $L$. 
Let $A$ be the $\SfS 7$ sublattice of~$L[S]$ associated with $S$.
Then $A$ is minimal if{}f $S$ is distributive.
\end{lemma}

Let $L$ be an SPS lattice 
and let $S$ be a distributive covering square.
We define in $L[S]$ an equivalence relation $\bgg(S)$ as follows.
All equivalence classes of $\bgg(S)$ are singletons except for the following:
\begin{align}\label{E:seti}
\set{z_{l,1}, x_{l,1}},& \dots, \set{z_{l,n}, x_{l,n}};\\ 
&\set{t, i}; \label{E:setii}\\
\set{z_{r,1}, x_{r,1}},& \dots, \set{z_{r,m}, x_{r,m}}.\label{E:setiii}
\end{align}

\begin{lemma}\label{L:distr}
Let $L$ be an SPS lattice 
and let $S$ be a distributive covering square. 
Then $\bgg(S)$ is a congruence relation of $L[S]$. 
\end{lemma}

\begin{proof}

By Lemma~\ref{L:technical}, 
we have to verify (C${}_{\jj}$) and (C${}_{\mm}$).
To verify (C${}_{\jj}$), let 
$x$ be covered by $y \neq z$ in $L[S]$ 
and $\cng x=y(\bgg(S))$. If $x = z_{l,i}$ and $y = x_{l,i}$, 
where $1 < i \leq n$, then $z = x_{l,i-1}$ and
$\cng {z= x_{l,i-1}}= {y \jj z = x_{l,i-1}}(\bgg(S))$, 
because $\set{z_{l,i}, x_{l,i}}$ is in the list \eqref{E:seti}.
If $i = 1$, we proceed the same way with $z = t$ 
and the list.
We proceed ``on the right'' 
with the lists \eqref{E:set2} and \eqref{E:setiii}.
Finally, $x = t$ cannot happen because $t$ is covered only by one element.

The verification of (C${}_{\mm}$) is very similar.
\end{proof}

Note that we use implicitly that an element of $\fil{y}$
cannot be covered by three elements, see Lemma~\ref{L:known}(i).
Also an element of $\fil{i}$ cannot cover three elements,
because $\fil{i}$ is a planar distributive lattice.

Lemma \ref{L:distr} is closely related 
to the construction of the lattice $K_f$ from
$K_{f-1}$ in G. Cz\'edli~\cite[p. 339]{gC12}, 
where a fork is inserted into a distributive covering cell.

\section{Protrusions}

Let $L$ be an SPS lattice 
and let $S$ be a covering square~$S = \set{o, a_l, a_r, i}$ of~$L$.
We call $S$ a \emph{tight square} 
if $i$ covers exactly two elements, namely, $a_l$ and $a_r$, in $L$;
otherwise, $S$ is a \emph{wide square}.

For a tight square $S$, a \emph{protrusion} is an element
$x_{l,i}$, $i = 1, \dots, n$ 
(or symmetrically, $x_{r,i}$, $i = 1, \dots, m$),
such that $x_{l,i}$ covers three or more element in $L$
(equivalently, three or more element in $L[S]$); 
see Figure~\ref{F:protrusion}, 
especially for the notation $a_{l,i+2}$, shown for $i = 2$, 
for the element of $L$ covered by $x_{l,i}$ 
that is immediately left of $x_{l,i+1}$.

If $x_{l,i}$ is a protrusion and $x_{l,i+3}$ has two covers, 
we denote by $a_{l,i+3}$ the cover different from $x_{l,i+2}$; 
we define $a_{l,i+4}$, and so on, similarly. 
Let $i^*$ be the largest integer for which $a_{l,i*}$ is defined.
The elements $x_{l,i+2}$, \dots, $x_{l,i^*}$ are the \emph{extensions} 
of the protrusion. 
There is always one extension, namely, $x_{l,i+2}$. 
In Figure~\ref{F:protrusion}, 
$i = 2$ and $i^* = 6$, there are three extensions.

Let $P_l \ci \set{1,\dots, n}$ denote the set of all left protrusions,
that is $i \in P_l$ if{}f $x_{l,i}$ is a left-protrusion. 
We define $P_r$ symmetrically.

\begin{figure}[htp]
  \centerline{\includegraphics{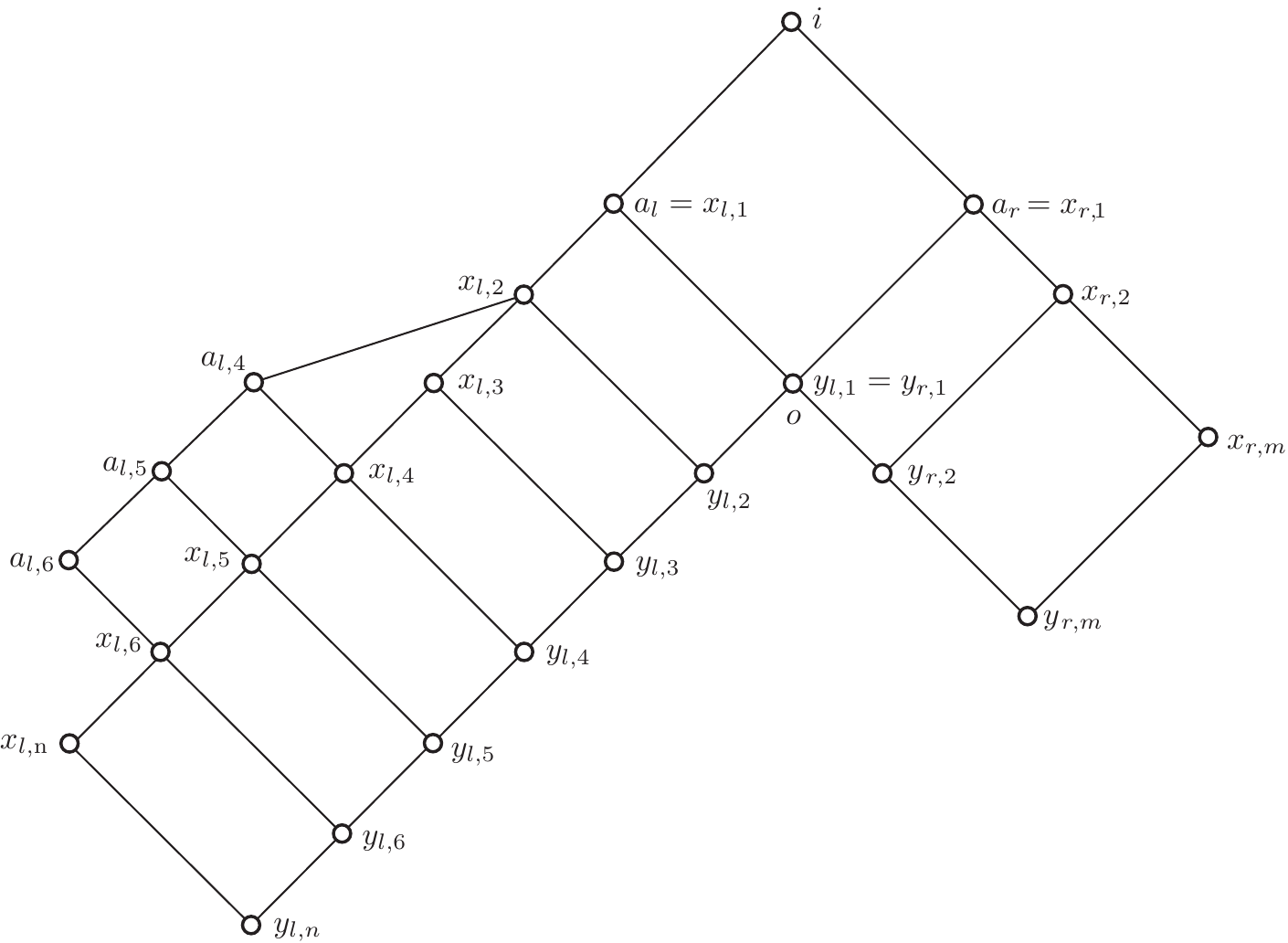}}
  \caption{Protrusion at $x_{l,2}$ and extensions $x_{l,4}$,
  $x_{l,5}$, $x_{l,6}$}
  \label{F:protrusion}
\end{figure}

\begin{figure}[b]
 \centerline{\includegraphics[scale=0.85]{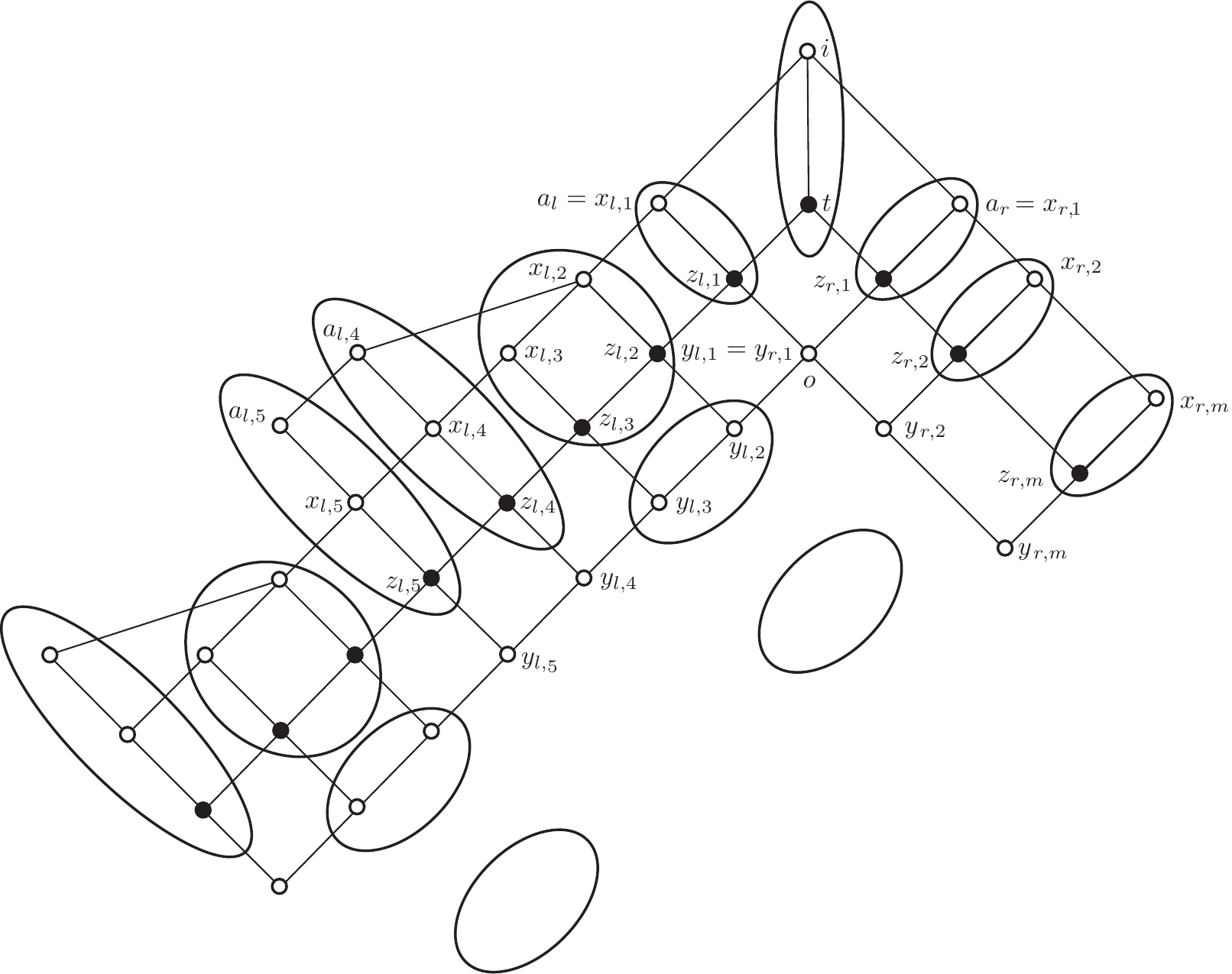}}
\caption{The congruence $\bgg(S)$ with two protrusions}\label{F:gamma}
\end{figure}

For every $i \in P_l$, define 
$\bgp_{l,i} = \consub{L}{y_{l,i}, y_{l,i+1}}$ and
$\bgp_l = \JJm{\bgp_{l,i}}{i \in P_l}$.

We define, symmetrically, $\bgp_{r,i}$ and $\bgp_r$. Finally,
let $\bgp = \bgp_l \jj \bgp_r$, the \emph{protrusion congruence}. 

We will also need the congruence $\ol\bgp_{l,i}$ and $\ol\bgp$,  
the congruences of $L[S]$ generated by $\bgp_{l,i}$ and $\ol\bgp$,
respectively. 

\begin{lemma}\label{L:pi}
Let $L$ be an SPS lattice and let $S$ be a tight square.
Let $i \in P_l$. 
In $L[S]$, all nontrivial congruence classes of $\ol\bgp_{l,i}$ 
except for $\set{{x_{l,i}}, x_{l,i+1}}$
and $\set{{y_{l,i}}, y_{l,i+1}}$ 
are in $\fil{y_{l,2} \mm y_{r,2}}$.
\end{lemma}

\begin{proof}
By the Structure Theorem (Theorem \ref{T:Structure Theorem}), 
the lattice $L[S]$ can be obtained 
from a planar distributive lattice $D$
by inserting $n$ forks. We induct on $n$. 

If $n=1$, there is no protrusion, so there is no $\bgp_{l,i}$. There is nothing to prove.

Let us assume that the statement holds for $n-1$. 
Let $A = \set{o^A, a_l^A, a_r^A, i^A}$ be a covering square of $L[S]$ such that $[o^A, i^A]$ is minimal $\SfS 7$ sublattice of~$L[S]$. 
Let $L_1 = L[S] - F[A]$. 
Then $L_1$ is a sublattice of $L[S]$
and $L_1[A] = L[S]$.

We can obtain $L_1$ from the planar distributive lattice $D$
by inserting $n-1$ forks. 
So the statement holds for $L_1$.
By inserting the fork at $A$, may make the congruence $\ol\bgp_{l,i}$ larger 
in $\fil{o^A}$ in $L_1[A] = L[S]$, 
but on the upper left and upper right boundaries of $\fil{i_A}$
there is no change, so there is no effect 
outside of the filter $\fil{y_{l,2} \mm y_{r,2}}$.
\end{proof}

\section{Tight squares}

\begin{lemma}\label{L:protrusion}
Let $L$ be an SPS lattice, $S$ a tight square of $L$.
Let $i \in P_l$. 
Then 
\begin{equation}\label{E:protrusioncong}
   \ol\bgp_{l,i} = \con{y_{l,i}, y_{l,i+1}} \leq \ol\bga_r(S).
\end{equation}

\end{lemma}
\begin{proof}
Computing in the $\SfS 7$ sublattice 
generated by $\set{x_{l,i},y_{l,i},a_{l,i+2}}$, 
we get that 
\[
   \con{x_{l,i}, x_{l,i+1}} \leq \con{x_{l,i}, y_{l,i+1}}.
\]
Since 
\[
   \con{a_r, i} = \con{y_{l,i}, y_{l,i+1}},
\]
\[
   \con{x_{l,i}, x_{l,i+1}} = \con{y_{l,i}, y_{l,i+1}},
\]
the statement follows.
\end{proof}

Now comes the crucial definition of the equivalence relation 
$\bgg_S$ on $L[S]$:

Define 
\begin{equation}
   i/\bgg_S = \set{i, t}.\label{E:i}
   \end{equation}

If $x_{l,i}$ and $x_{l,i-1}$ are not protrusions 
and $x_{l,i}$ is not an extension,
\begin{equation}
   x_{l,i}/\bgg_S = \set{x_{l,i}, z_{l,i}}.\label{E:ii}
\end{equation}

If $x_{l,i}$ is a protrusion,
\begin{align}
   x_{l,i}/\bgg_S 
   &= \set{x_{l,i}, x_{l,i+1},z_{l,i},z_{l,i+1}};\label{E:iii}\\
   y_{l,i}/\bgg_S &= \set{y_{l,i}, y_{l,i+1}};\label{E:iv}\\
   \intertext{if $x_{l,j}$ is an extension, 
   that is, $i+2 \leq j \leq i^*$, then }
    x_{l,i+2}/\bgg_S 
   &= \set{a_{l,i+2}, x_{l,i+2}, z_{l,i+2}}.\label{E:v}
\end{align}

We define $x_{r,i}/\bgg_S$, $y_{r,i}/\bgg_S$, 
and $x_{r,i+2}/\bgg_S$ symmetrically.

Let $x/\bgg_S$ be as defined above; let  
\begin{equation}\label{E:vi}
   x/\bgg_S = x/\bgp,
\end{equation}
otherwise, see Figure~\ref{F:gamma}.

Clearly, $\bgp \leq \bgg_S$.

\begin{lemma}\label{L:gamma}
Let $L$ be an SPS lattice. 
Let $S$ be a \emph{tight square}. 
Then $\bgg_S = \bgg(S)$ is a congruence relation on $L[S]$.
\end{lemma}

\begin{proof}
Since $\bgg_S$ is an equivalence relation 
with intervals as equivalence classes,
by Lemma \ref{L:technical}, 
we only have to verify (C${}_{\jj}$) and (C${}_{\mm}$).

\emph{To verify} (C${}_{\jj}$),
let $v \neq w$ cover $u$ and let $\cng u = v (\bgg_S)$.
Then we distinguish six cases according to \eqref{E:i}--\eqref{E:vi}
in the definition of $\bgg_S$.

Case \eqref{E:i}: $u = t$. This cannot happen 
because $t$ has only one cover.

Case \eqref{E:ii}: $u = z_{l,i}$, $v = x_{l,i}$. 
Then $w = z_{l,i-1}$ and so 
\[
   \cng {w = z_{l,i-1}} = {v \jj w = x_{l,i-1}} (\bgg_S)
\]
by \eqref{E:ii}.  

Case \eqref{E:iii}: $x_{l,i}$ is a protrusion and 
$u,v \in \set{x_{l,i}, x_{l,i+1},z_{l,i},z_{l,i+1}}$.
 If $u = z_{l,i}$, $v = x_{l,i}$
or if $u = z_{l,i+1}$, $v = x_{l,i+1}$, 
we proceed as in Case \eqref{E:ii}. 
$u = x_{l,i+1}$ cannot happen because $x_{l,i+1}$ has only one cover. 
So we are left with $u = z_{l,i+1}$, $v = z_{l,i}$. 
Then $w = x_{l,i+1}$ and 
\[
   \cng {w = z_{l,i+1}} = {v \jj w = x_{l,i}} (\bgg_S).
\]

Case \eqref{E:iv}: $x_{l,i}$ is a protrusion and 
$u,v \in \set{y_{l,i}, y_{l,i+1}}$. 
Then $u = y_{l,i+1}$, $v = y_{l,i}$. Then $w = x_{l,i+1}$ and 
\[
   \cng {w = z_{l,i+1}} = {v \jj w = x_{l,i}} (\bgg_S)
\]
by \eqref{E:iii}.

Case \eqref{E:v}: $x_{l,i}$ is a protrusion and 
$u,v \in \set{a_{l,i+2}, x_{l,i+2}, z_{l,i+2}}$.
In this case, either $u = z_{l,i+2}$, $v = x_{l,i+2}$
or $u = x_{l,i+2}$, $v = a_{l,i+2}$. 
If $u = z_{l,i+2}$, $v = x_{l,i+2}$, then $w = z_{l,i+1}$ and 
\[
   \cng {w = x_{l,i+1}} = {v \jj w = x_{l,i+1}} (\bgg_S)
\]
by \eqref{E:iii}. 
If $u = x_{l,i+2}$, $v = a_{l,i+2}$, then $w = x_{l,i+1}$ and
\[
   \cng {w = x_{l,i+1}} = {v \jj w = x_{l,i}} (\bgg_S)
\]
by \eqref{E:iii}.

Case \eqref{E:vi}: $u, v \in x/\bgp$. 
Then  $\cng w = v \jj w (\bgp)$ since $\bgp$ is a congruence.

\emph{To verify} (C${}_{\mm}$), let $u$ cover $v \neq w$ 
and let $\cng v = u (\bgg_S)$. Then we again distinguish six cases.

Case \eqref{E:i}: $v = t$, $u = i$. 
Since $S$ is tight, $w = x_{l,1}$ or symmetrically. 
Then 
\[
   \cng {w = x_{l,1}} = {v \mm w = z_{l,1}} (\bgg_S)
\]
by \eqref{E:ii}.

Case \eqref{E:ii}: $v = z_{l,i}$, $u = x_{l,i}$. 
Then $w = x_{l,i-1}$ and so 
\[
   \cng {w = x_{l,i+1}} = {v \mm w = z_{l,i+1}} (\bgg_S)
\]
by \eqref{E:ii}.  

Case \eqref{E:iii}: $x_{l,i}$ is a protrusion and 
$u,v \in \set{x_{l,i}, x_{l,i+1},z_{l,i},z_{l,i+1}}$.
If $u = z_{l,i}$, $v = x_{l,i}$
or if $u = z_{l,i+1}$, $v = x_{l,i+1}$, 
we proceed as in Case \eqref{E:ii}. 
If $u = x_{l,i}$, $v = x_{l,i+1}$, then $w = z_{l,i}$ and
\[
   \cng {w = z_{l,i}} = {v \mm w = z_{l,i+1}} (\bgg_S)
\]
by \eqref{E:iii}.
Finally, let $u = z_{l,i}$, $v = z_{l,i+1}$. 
Then $w = y_{l,i}$ and 
\[
   \cng {w = y_{l,i}} = {v \jj w = y_{l,i+1}} (\bgg_S)
\]
by \eqref{E:iv}.

Case \eqref{E:iv}: $x_{l,i}$ is a protrusion and 
$u,v \in \set{y_{l,i}, y_{l,i+1}}$. 
Then $u = y_{l,i}$, $v = y_{l,i}$ and 
$
   \cng w = {v \mm w = y_{l,i+1}} (\bgg_S)
$
by \eqref{E:vi}.

Case \eqref{E:v}: $x_{l,i}$ is a protrusion and 
$u,v \in \set{a_{l,i+2}, x_{l,i+2}, z_{l,i+2}}$. 
If $u = x_{l,i+2}$, $v = z_{l,i+2}$, 
we proceed as in Case \eqref{E:ii}.
Otherwise, $u = a_{l,i+2}$, $v = x_{l,i+2}$ and
\[
   \cng {w = x_{l,i+1}} = {v \mm w = x_{l,i}} (\bgg_S)
\]
by \eqref{E:iii}.

Case \eqref{E:vi}: $u, v \in x/\bgp$. 
Then  $\cng w = v \mm w (\bgp)$ since $\bgp$ is a congruence.
\end{proof}

We also need to describe $\ol\bgp_{l,i}$.
\begin{lemma}\label{L:}
Let $L$ be an SPS lattice, $S$ a tight square of $L$.
Let $i \in P_l$. 
Then the congruence classes of $\ol\bgp_{l,i}$ in $L[S]$ 
are the congruence classes of $\bgp_{l,i}$ in $L$
and
\begin{equation}\label{E:olp}
   \set{x_{l,i}, x_{l,i+1}}.
\end{equation}
Moreover, neither $\bgp_{l,i}$ nor $\ol\bgp_{l,i}$ 
has any nontrivial congruence classes outside of~$\fil{i}$. 
In particular, $\ol\bgp_{l,i} < \bgg(S)$.
\end{lemma}

\begin{proof}
Trivial.
\end{proof}

Let $G$ denote the set of prime intervals of $L[S]$ listed in 
\eqref{E:seti}--\eqref{E:setiii}. 

\begin{lemma}\label{L:G}
Let $L$ be an SPS lattice. 
Let $S$ be a \emph{tight square}. 
For a prime interval $\fp$ of $L[S]$, 
we have $\bgg(S) = \con{\fp}$ if{}f $\fp \in G$.
\end{lemma}

\begin{proof}
If we look at all the prime intervals $\fp$ collapsed by $\bgg(S)$
as listed in \eqref{E:i}--\eqref{E:vi}, then they are either 
listed in \eqref{E:seti}--\eqref{E:setiii} 
or they generate $\ol\bgp_{l,i}$.
\end{proof}

\begin{theorem}\label{T:tight}
Let $L$ be an SPS lattice. 
Let $S$ be a \emph{tight square}. 
Let $\fp = [u,v]$ be a prime interval of $L$.
Let us assume that $\bgg(S) < \con \fp$
in $L[S]$. Then either $\bga_l(S) \leq \con \fp$
or $\bga_r(S) \leq \con \fp$ in~$L$.
\end{theorem}

\begin{proof}
So let $\bgg(S) < \con \fp$ in $L[S]$ 
for a prime interval $\fp$ of $L[S]$. 
By~Lemmas \ref{L:cproj} and \ref{L:G}, 
there is a sequence of intervals in $L[S]$:
\begin{equation}\tag{S}
   \fp = \inv{e_0}{f_0} \cpersp \inv{e_1}{f_1} \cpersp 
      \cdots \cpersp \inv{e_n}{f_n} = \fq= [x_{l,i}, y_{l,i}] \in G.
\end{equation}
(or symmetrically), using the notation $t = x_{l,0}$, $i = y_{l,0}$.
We can assume that (S) was chosen to minimize $n$. 
In particular, $\cperspup$ and $\cperspdn$ alternate in~(S).


The interval $[e_{n-1}, f_{n-1}]$ 
has a prime subinterval $[e'_{n-1}, f'_{n-1}]$ perspective to~$\fq$. 
Since $G$ is a set of prime interval 
closed under prime perspectivity, 
it follows that $[e_{n-1}, f_{n-1}]$ has $\fq$ as a subinterval
by minimality.

We cannot have $[e'_{n-1}, f'_{n-1}] = [e_{n}, f_{n}]$ 
because this conflicts with the minimality of $n$. So there are two cases to consider.

First, let $e_{n-1} < e_n$ and $[e_{n-1}, f_n] \in G$. 
Then $f_{n-1} = f_n$.
If $[e_n, f_n] = [t,i]$, 
then $e_{n-1} \leq z_{l,1}$ (or symmetrically),
therefore, $\cng z_{l,1} = t(\con{\fp})$ 
and so $\con{\fp} \geq \bga_l$, as claimed.

If $[e_n, f_n] = [z_{l,i}, x_{l,i}]$ (or symmetrically), 
with $i \geq 1$, then $e_{n-1} < z_{l,i}$ 
and so $e_{n-1} \leq z_{l,i+1}$
or $e_{n-1} \leq y_{l,i}$. The first possibility contradicts the minimality of $i$, while the second yields that 
$\con{\fp} \geq \bga_r$, as claimed.

Second, let $e_{n-1} = e_n$ and $[e_{n-1}, f_n] \in G$.
There are two possibilities: 
\begin{align}
   \inv{e_{n-2}}{f_{n-2}} &\cperspup [{e_{n-1}}{f_{n-1}}]\label{E:up}\\
\intertext{and}
   \inv{e_{n-2}}{f_{n-2}} &\cperspdn [{e_{n-1}}{f_{n-1}}]\label{E:dn}
\end{align}
If \eqref{E:up} holds, 
then $e_{n-2} \leq x_{l,i} < x_{l,i} < f_{n-2}$
and $f_{n-2} \jj e_{n-1} = f_{n-1}$
(which implies that $x_{l,i}$ precedes the first protrusion),
contradicting the minimality of $n$.

If \eqref{E:dn} holds, then $e_{n-2} = y_{l,i}$ with $i \geq 1$ or $e_{n-2} = t$. 
In either case, $[{e_{n-2}}{f_{n-2}}]$ contains a member of $G$,
contradicting the minimality of $n$.
\end{proof}

\begin{corollary}\label{C:tight}
Let $S$ be a \emph{tight square} 
in an SPS lattice $L$.
Then the congruence~$\bgg(S)$ of~$L[S]$ is covered by
one or two congruences in $\Ji (\Con L[S])$, 
namely, by $\ol\bga_l(S)$ and $\ol\bga_r(S)$.
\end{corollary}

\end{document}